\documentclass[11pt,reqno,a4paper]{amsart}

\usepackage{amsmath, amssymb, amsthm}
\usepackage{amsfonts}
\usepackage{graphicx}
\makeatother
\usepackage{hyperref}
\usepackage{geometry}
\usepackage{enumerate}
\usepackage{etoolbox}
\usepackage{amscd}
\usepackage{bold-extra}
\usepackage[all,cmtip]{xy}
\usepackage[latin1]{inputenc}

\patchcmd{\subsection}{-.5em}{.5em}{}{}
\patchcmd{\subsubsection}{-.5em}{.5em}{}{}

\numberwithin{equation}{section}

\newcommand{\SL}{\operatorname{SL}}

\newcommand{\GL}{\operatorname{GL}}

\newcommand{\Aut}{\operatorname{Aut}}

\newcommand{\cB}{\mathcal{B}}

\newcommand{\cH}{\mathcal{H}}

\newcommand{\bE}{\mathbb{E}}

\newcommand{\bN}{\mathbb{N}}

\newcommand{\bP}{\mathbb{P}}
\newcommand{\bQ}{\mathbb{Q}}
\newcommand{\bR}{\mathbb{R}}

\newcommand{\bT}{\mathbb{T}}

\newcommand{\bZ}{\mathbb{Z}}

\numberwithin{equation}{section}

\newtheorem{thm}{Theorem}[section]

\newtheorem{lem}[thm]{Lemma}

\newtheorem*{iprob*}{Problem}

\theoremstyle{definition}

\newtheorem{defi}[thm]{Definition}
\newtheorem{rem}[thm]{Remark}

\DeclareMathOperator{\Vol}{Vol}
\DeclareMathOperator{\VS}{VolSpec}
\DeclareMathOperator{\ES}{EhrSpec}
\DeclareMathOperator{\Rat}{Rat}

\title{Ehrhart spectra of large subsets in $\bZ^r$}
\date{}

\author[Michael Bj\"orklund]{Michael Bj\"orklund}
\address{Department of Mathematics, Chalmers and University of Gothenburg, Gothenburg, Sweden}
\email{micbjo@chalmers.se}

\author[Rickard Cullman]{Rickard Cullman}
\address{Department of Mathematics, Chalmers and University of Gothenburg, Gothenburg, Sweden}
\email{cullman@chalmers.se}

\author[Alexander Fish]{Alexander Fish}
\address{School of Mathematics and Statistics, University of Sydney, NSW 2006, Australia}
\email{alexander.fish@sydney.edu.au}


\begin{document}
\begin{abstract}
This paper introduces and studies the Ehrhart spectrum of a set $E \subseteq \bZ^r$, defined as the set of all Ehrhart polynomials of simplices with vertices in $E$, generalizing the notion of volume spectrum. We show that for any $E \subseteq\bZ^r$ with positive upper Banach density, there is some $n \in \bZ$ such that the Ehrhart spectrum of $n \bZ ^r$ is contained in the Erhard spectrum of $E$, generalizing an earlier result by the first and third author for the volume spectrum of $E$. 
\end{abstract}
\maketitle
\thispagestyle{empty}
%

\section{Introduction}
Let $\beta:= \{ v_0, ..., v_r \} \subseteq \bZ^r$ be a set of $r+1$ affinely independent points, i.e 
\begin{equation*}
    v_1-v_0, ...v_r - v_0
\end{equation*}
are linearly independent. To such a set, we may associate the \emph{simplex} $\Delta_\beta \subseteq \bR^r$   of all convex combinations of points in $\beta$, that is
\begin{equation*}
    \Delta_\beta:= \left\{ v_0 + \sum_{k=1}^r t_k (v_k - v_0): t_k \in [0,1] \mbox{ for } k=1,..., r \mbox{ and } \sum_{k=1}^r t_k \le 1 \right\}.
\end{equation*}

In \cite{BF}, the first and third authors introduced and studied the \emph{volume spectrum} 
\begin{equation*}
    \VS(E) := \left\{\Vol(\Delta_\beta): \beta \subseteq E \right\} \subseteq \frac{1}{r!} \bZ
\end{equation*}
of a set $E \subseteq\bZ^r$.
 To explain their result, we recall that the \emph{upper Banach density} $d^* (E)$ is defined as
\begin{equation*}
    d^* (E) := \limsup_{n \to \infty} \sup_{v \in \bZ^r} \frac{|E \cap \left( v+ [0,n-1]^r\right)|}{n^r}.
\end{equation*}

\begin{thm}[\cite{BF}, Corollary 1.2] \label{volspec}
    If $E \subseteq \bZ^r$ has positive upper Banach density, there is some $n \in \bN$ such that 
    \begin{equation*}
        n \cdot (\bZ \setminus \{0\} ) \subseteq r! \cdot \VS (E).
    \end{equation*}
\end{thm}
In addition, the third author and Sean Skinner proved in \cite[Theorem B]{SF} that $n$ may be chosen to depend only on $d^* (E)$. Crucial to this result is the fact that the volume, as a function on simplices in $\bZ^r$, is invariant under the action of $\SL_r (\bZ)$ on $\bZ^r$. This leads naturally to the question of whether similar results hold for other $\SL_r (\bZ)$-invariant valuations on the set of simplices. \\

As we explain below, essentially all information about such valuations is encoded in the \emph{Ehrhart polynomial} $P_\Delta$ associated to a simplex $\Delta$ with all the vertices being on the integer lattice $\bZ^r$, defined as
\begin{equation*}
    P_\Delta (t) := |t \cdot \Delta \cap \bZ^r|, \quad \textrm{for $t \in \bN$}.
\end{equation*}
Ehrhart proved in \cite{Ehrhart} that $P_\Delta$ indeed takes the form of a polynomial 
\begin{equation*}
    L(\Delta,  r) t^r + L(\Delta, r-1) t^{r-1} + ... + L(\Delta, 1) t + L(\Delta, 0) \quad \textrm{for all $t \in \bN$},
\end{equation*}
and this polynomial only depends on the $\SL_r (\bZ)$-orbit of the simplex $\Delta$. Indeed, for any $\gamma \in SL_r(\bZ)$, we have 
\[
P_{\gamma(\Delta)}(t) =  |t \cdot \gamma(\Delta) \cap \bZ^r| = |t \cdot \Delta \cap \gamma^{-1}(\bZ^r)| =  |t \cdot \Delta \cap \bZ^r| =  P_\Delta (t). 
\]
Furthermore, it does not change if $\Delta$ is shifted by a vector in $\bZ^r$. Thus the coefficients
\begin{equation*}
    \Delta \mapsto L(\Delta, k)
\end{equation*} 
for $k=0, ..., r$ are both $\SL_r (\bZ)$- and translation-invariant functions; for example, 
the leading coefficient $L( \Delta, r)$ is, up to a fixed multiplicative constant, the volume of $\Delta$. Furthermore, it was shown in \cite[Korollar 3]{KneserBetke} that any $\SL_r (\bZ)$-invariant valuation on the set of simplices is a linear combination of Ehrhart coefficients $L( \cdot, k)$ for $k = 0, ..., r$. \\

In light of this, we introduce the \emph{Ehrhart spectrum} of a set $E \subseteq \bZ^r$ as a generalization of the volume spectrum. 

\begin{defi}[Ehrhart spectrum]
The \emph{Ehrhart spectrum} $\ES(E) \subset \bQ[t]$ of a set $E \subset \bZ^r$
is defined as the collection of all Ehrhart polynomials for simplices with vertices in $E$. 
\end{defi}

The primary goal of this paper is to extend Theorem \ref{volspec} by incorporating the full Ehrhart spectrum.

\begin{thm}\label{ehrharttheorem}
    Let $E \subseteq \bZ^r$ be a set with positive upper Banach density. Then there is some $n \in \bN$ such that 
    \begin{equation*}
  \ES (n \bZ^r) \subseteq \ES (E). 
    \end{equation*}
\end{thm}
Since the Ehrhart polynomial of a simplex is invariant under both translation in $\bZ^r$ and the action of $\SL_r (\bZ)$, Theorem \ref{ehrharttheorem} follows directly from the following result. 

\begin{thm} \label{maincombinatorial}
    Let $E \subseteq \bZ^r$ be a set with positive upper Banach density. There exists an $n \in \bN$ such that for any set $v_1, ..., v_r \in \bZ^r$ of linearly independent vectors, there is $v_0 \in E$ and $\gamma \in \mbox{SL}_r (\bZ)$ such that
    \begin{equation*}
        v_0 + n \gamma (v_k) \in E
    \end{equation*}
    for $k = 1,..., r$.
\end{thm}

\begin{rem}
Theorem \ref{maincombinatorial} implies Theorem \ref{volspec} by considering the leading coefficient of the Ehrhart polynomials on both sides of the inclusion.
\end{rem}

\begin{proof}[Proof of Theorem \ref{ehrharttheorem} using Theorem \ref{maincombinatorial}]
    Let $\Delta_\beta$ be any simplex in $\bZ^r$ with vertex set $\beta = \{u_0, ..., u_r \}$ and Ehrhart polynomial $P_{\Delta_\beta}$. By translation invariance, we may assume $u_0 = 0$. By Theorem \ref{maincombinatorial} we can find $n \in \bN,$ $\gamma \in \SL_r (\bZ)$ and $v_0 \in E$ such that 
    \begin{equation*}
        v_0 + \gamma (n \beta) \subseteq E, 
    \end{equation*}
    so 
    \begin{equation*}
        \ES (E)  \ni P_{\Delta_{v_0 + \gamma (n\beta)}} = P_{\Delta_ {n \cdot \beta}} = P_{n \Delta_\beta}.
    \end{equation*}
    As $\Delta_\beta$ was an arbitrary simplex, we are done. 
\end{proof}

Theorem \ref{maincombinatorial} can be equivalently formulated as
\begin{equation*}
    E \cap (E - n \gamma(v_1)) \cap ... \cap (E - n \gamma (v_r)) \neq 0
\end{equation*}
for some $n$ and $\gamma$, which in turn is, via the Furstenberg correspondence principle, a straightforward consequence of the following dynamical theorem, to whose proof this paper is devoted. 

\begin{thm}\label{maindynamical}
    Let $(X, \mu)$ be an ergodic probability measure preserving $\bZ^r$-space and $B \in \cB_X$ a set of positive measure. There exists an $n \in \bN$ such that for any set $v_1, ..., v_r \in n \bZ^r$ of linearly independent vectors, there is $\gamma \in \SL_r (\bZ)$ such that
    \begin{equation*}
        \mu (B \cap \gamma(v_1).B \cap ...\cap  \gamma(v_r).B) > 0. 
    \end{equation*}
\end{thm}

\begin{rem}
We emphasize that this is a genuine multi-correlation result, as the same  
element $\gamma \in \SL_r(\bZ)$ is used throughout the multiple correlation. This sharply contrasts with earlier results \cite{BB, BFChar, BFP, F} by the first and third authors, in collaborations with Bulinski and Parkinson, where ensuring positivity required the use of different $\gamma$'s.
\end{rem}

\medskip

\subsection*{Acknowledgments} 
M.B. was supported by the grant 11253320 from the Swedish Research Council. A.F. would like to thank the support of the Australian Research council via the grant DP240100472. 
A.F. is grateful for the generous hospitality of Chalmers University in Gothenburg during February 2025, when this work was completed.

\section{General setup}
In this paper, we denote by $\Lambda$ a free $\bZ$-module generated by $r$ elements, i.e., $\Lambda \cong \bZ^r$.  For $n \in \bN$, let $\Lambda(n) := n! \cdot \Lambda$
and note that $\Lambda(n)$ is also a free $\bZ$-module with $r$ generators, which has 
finite index in $\Lambda$. We denote by $\Aut_1 (\Lambda)$ the group of all automorphisms of $\Lambda$ with unit determinant. Note that 
$\Aut_1 (\Lambda) \cong \SL_r (\bZ^r)$. We denote by $\widehat{\Lambda}$ the Pontryagin
dual of $\Lambda$, and write 
\[
\langle \cdot,\cdot \rangle : \widehat{\Lambda} \times \Lambda \to \bT, \enskip (\xi,v) \mapsto \langle \xi,v \rangle,
\]
for the canonical pairing. \\

We collect below some concepts and notation to be used in our arguments. 

\begin{defi}[Rational spectrum]
    We define the \emph{rational spectrum} $\Rat\mbox{ } \Lambda \subseteq\widehat{\Lambda}$ as the set of all elements $\xi \in \widehat{\Lambda}$ such that $\ker \xi$ has finite index in $\Lambda$. 
\end{defi}

\begin{defi}[Spectral measure] \label{spectraldef}
    Let $(X, \mu)$ be a measure-preserving $\Lambda$-space, and let $B \in \cB_X$ be a set of positive measure. By Bochner's theorem there is a \emph{spectral measure} $\sigma_{\mu, B}$ on $\widehat{\Lambda}$ such that
    \begin{equation*}
        \mu(B \cap v. B) = \int_{\widehat{\Lambda}} \langle \xi, v\rangle d \sigma_{\mu, B} (\xi). 
    \end{equation*}
\end{defi}

We often suppress the dependence on the measure $\mu$ in $\sigma_{\mu, B}$ and write $\sigma_B$ if no ambiguity arises. The spectral measure $\sigma_B$ can be decomposed as 
\begin{equation} \label{spectraldecomp}
    \mu (B) ^2 \delta_0 + \sigma_B ^{\Rat} + \tau
\end{equation}
where $\sigma_B ^{\Rat}$ is the restriction of $\sigma_B$ to $\mbox{Rat }\Lambda \setminus\{0\}$ and $\tau$ the restriction to $\widehat{\Lambda} \setminus\Rat  \mbox{ } \Lambda$. 

\begin{defi} \label{Sdef}
    Let $\beta := (v_1, ..., v_r)$ be an ordered basis of $\Lambda$. We define the set of $r - 1$ linear transformations $\{S^\beta _l \}_{l=2}^r \subseteq \Aut_1 (\Lambda)$ in the basis $\beta$ by the following action on the basis elements:
    \begin{equation*}
    S^\beta _l (v_k) = v_k + \delta_{kl} v_1. 
    \end{equation*}
\end{defi}

\begin{rem}
If $\beta$ is a basis of a finite index subgroup in $\Lambda$, then the maps defined by Definition \ref{Sdef} can be extended by linearity to all $\Lambda$ and the obtained extended maps will be in $\Aut_1 (\Lambda)$.
\end{rem}

\begin{defi}[Haystack] \label{haystackdef}
    An infinite set $\cH \subseteq \Lambda$ is a \emph{haystack} if any set of $r$ elements in $\cH$ generates a finite-index subgroup of $\Lambda$. Equivalently, any set of $r$ elements in $\cH$ are linearly independent over $\bR$. 
\end{defi}

\begin{rem}
    In their definition of a haystack in \cite{BF}, the first and third authors also require $\cH$ to be a subset of the primitive vectors in $\Lambda$. This, assumption, however, is never used in the proofs of Lemmas \ref{haystackexpansion} and \ref{ergodiccomponent} below, so we omit it from our definition. 
\end{rem}

\begin{defi}
A group $\Gamma < \GL_r (\bZ)$ is called \textit{strongly irreducible} if every finite index subgroup of $\Gamma$ acts irreducibly on $\bR^r$. An element $\gamma \in \Gamma$ is called \textit{proximal} if $\gamma$  has only one eigenvalue of the largest absolute value and the corresponding eigenspace is one-dimensional.
\end{defi}

\begin{rem}\label{proximalelement}
The group $\Gamma = SL_r(\bZ)$ is strongly irreducible and contains proximal elements.
\end{rem}

\section{Proof of Theorem \ref{maindynamical}}

The key part of the argument is the following Lemma; the proof of it comprises the main part of this paper.  

\begin{lem} \label{haystack}
    Let $(X, \mu)$ be an ergodic $\Lambda$-space, let $B \in \cB_X$ be a set of positive measure and let $\sigma_{\mu, B}$ be its spectral measure. Let $\Gamma < \GL_r (\bZ)$ be a finitely generated strongly irreducible subgroup containing a proximal element. Let $0 < c < \mu(B)^2 - \sigma_{\mu, B} (\Rat \mbox{ } \Lambda \setminus\{0\})$. Then for any $v_0 \in \Lambda$ the set
    \begin{equation*}
        \{v \in \Lambda: \mu(B \cap v.B) > c \} \cap \Gamma .v_0
    \end{equation*}
    contains a haystack. 
\end{lem}

The following two crucial results from \cite{BF} now allows us to derive Theorem \ref{maindynamical} using Lemma \ref{haystack}. Note that in \cite{BF}, the results are stated using the \emph{normalized spectral measure} $\widetilde{\sigma_B} := \frac{1}{\mu(B)^2} \sigma_B$, while here we use the standard spectral measure $\sigma_B$. The different formulations are equivalent, however.   
\begin{lem}[Theorem 3.1, \cite{BF}] \label{haystackexpansion}
     Let $(X, \mu)$ be an ergodic $\Lambda$-space, let $B \in \cB_X$ be a set of positive measure and let $\sigma_{\mu, B}$ be its spectral measure.  Let $\sigma_B (\mbox{Rat } \Lambda \setminus \{0\}) < \epsilon \mu(B) ^2$.  Then for any haystack $\cH \subseteq \Lambda$ there is $\lambda_\epsilon \in \cH$ such that
    \begin{equation*}
        \mu (\bZ \lambda_\epsilon . B) > 1 - \epsilon.  
    \end{equation*}
\end{lem}

\begin{lem}[Proposition. 4.1, \cite{BF}] \label{ergodiccomponent}
    For every $\epsilon > 0$, there is an integer $n$, a positive constant $c$ and a $\Lambda(n)$-invariant and ergodic probability measure $\nu$  on $X$ such that 
    \begin{itemize}
        \item $\nu (B) > \frac{\mu(B)}{3}$
        \item $\sigma _{\nu, B} (\Rat\mbox{ }\Lambda(n) \setminus \{0\} ) < \epsilon$
        \item $\mu( \bigcap_{v \in F} v.B) \geq c \cdot \nu( \bigcap_{v \in F} v.B)$. 
    \end{itemize}
    for any finite subset $F$ of $\Lambda(n)$. 
\end{lem}

\begin{proof}[Proof of Theorem \ref{maindynamical}]
Let $v_1,\ldots,v_r \in \Lambda$ be linearly independent vectors and denote by $\Gamma = SL_r(\bZ)$. 
    Pick $\epsilon > 0$ such that $\epsilon < \frac{\mu(B)^2}{18(r-1)}$ and, using Lemma \ref{ergodiccomponent}, an $n \in \bZ$ and a $\Lambda(n)$-ergodic measure $\nu$ on $X$ such that 
    \begin{equation*}
        \sigma_{\nu, B} (\mbox{Rat } \Lambda(n) \setminus \{0\}) <  \epsilon \mu(B)^2, \mbox{ and } \nu(B) > \frac{\mu(B)}{3}.
    \end{equation*}
    Then 
    \begin{equation*}
        \begin{split}
            \nu(B)^2 - \sigma_{\nu, B} (\mbox{Rat } \Lambda(n) \setminus \{0\}) \geq \frac{\mu(B)^2}{9} - \sigma_{\nu, B} (\mbox{Rat } \Lambda(n) \setminus \{0\}) 
            > \frac{\mu(B)^2}{18}.
        \end{split}
    \end{equation*}
 Using Remark \ref{proximalelement}, $\Gamma$ acts on $\bR^r$ strongly irreducibly and with a proximal element. By Lemma \ref{haystack} applied to $\nu$ and $\Lambda(n)$ with $c = \frac{\mu(B)^2}{18}$ and $v_0 = n!v_1$, using Lemma \ref{haystackexpansion} there is $u_1 \in \Gamma. v_0$ such that $$\nu(B \cap u_1.B) > \frac{\mu(B)^2}{18}$$ and $\nu((\bZ u_1.B)^c) < \epsilon$. Now for any set of vectors $u_2, ..., u_r \in \Lambda(n)$, we have, since $\Lambda(n)$ preserves $\nu$, that $\nu(u_k.((\bZ u_1.B)^c) = \nu((\bZ u_1. B)^c) < \epsilon$, which gives 
    \begin{equation*}
        \begin{split}
            \nu( B \cap u_1. B \cap (u_2 + \bZ u_1).B \cap ... \cap (u_r + \bZ u_1).B ) \\
            =1 - \nu( (B \cap u_1. B)^c \cup u_2.(\bZ u_1.B)^c \cup ... \cup u_r .(\bZ u_1.B)^c ) \\
            \geq 1 - \nu((B \cap u_1.B)^c ) - \sum_{k=2}^r \nu(u_k.(\bZ u_1.B)^c) \\
            \geq \nu(B \cap u_1.B) - \sum_{k=2}^r \nu(u_k.(\bZ u_1.B)^c) \\
            \geq \frac{\mu(B)^2}{18} - (r-1)\epsilon > 0.
        \end{split}
    \end{equation*}
    Thus there are $m_2, ..., m_r \in \bZ$ such that 
    \begin{equation*}
        \nu(B \cap u_1.B \cap (u_2 + m_2 u_1).B \cap...\cap (u_r + m_r u_1).B) > 0, 
    \end{equation*}
    which implies, by the third point of Lemma \ref{ergodiccomponent}, that 
    \begin{equation*}
        \mu(B \cap u_1.B \cap (u_2 + m_2 u_1).B \cap...\cap (u_r + m_r u_1).B) > 0.
    \end{equation*}
    Since $u_1 \in \Aut_1 (\Lambda) . (n! v_1)$ there is some $\gamma_0 \in \Aut_1 (\Lambda)$ such that $\gamma_0 (n! v_1) = u_1$. Choosing $u_k = n! \gamma_0 (v_k)$ for $k = 2, ..., r$, we get 
    \begin{equation*}
        \begin{split}
            \mu(B \cap n! \gamma_0 (v_1).B \cap (n! \gamma_0 (v_2) + m_2 n! \gamma_0 (v_1)).B \cap...\cap (n! \gamma_0 (v_r) + m_r n! \gamma_0 (v_1)).B) \\
            = \mu( B \cap n! S \gamma_0 (v_1) . B \cap n! S \gamma_0 (v_2). B \cap ... \cap n! S\gamma_0 (v_r) . B) > 0, 
        \end{split} 
    \end{equation*}
    where $S := (S^\beta _2) ^{m_2} \cdots (S^\beta _r) ^{m_r}$ (recall Definition \ref{Sdef}), for $\beta = \{u_1, ..., u_r\}$. Note that since $\{v_1, ..., v_r\}$ are linearly independent, so are $u_1, ..., u_r$, and the $S^\beta _j$ are thus well-defined and contained in $\Aut_1 (\Lambda)$. Therefore $S \in \Aut_1 (\Lambda)$, so the same holds for $\gamma := S \gamma_0$, and we are done.   
\end{proof}

The rest of the paper is devoted to the proof of Lemma \ref{haystack}. 

\section{Proof of Lemma \ref{haystack}}
To prove the statement, we make use of the following lemma, which reduces the problem to showing that for $c > 0$ under consideration, the set
\begin{equation*}
    \{v \in \Lambda: \mu(B \cap v. B) > c \} \cap \Gamma .v_0
\end{equation*}
is not contained in a finite union of hyperplanes. 
\begin{lem} \label{haystackhyperplanes}
    Let $E \subseteq \bZ^r$, and suppose that $E$ is not contained in a finite union of hyperplanes in $\bR^r$. Then $E$ contains a haystack. 
\end{lem}
\begin{proof}
    Define the haystack $\cH \subseteq E$ as follows. Pick some nonzero $v_1 \in E$ and put $\cH_1 := \{v_1\}$. Supposing that we have defined $\cH_k$, pick some $v_{k+1} \in E$ which does not lie in any hyperplane spanned by $r-1$ elements in $\cH_k$; this is possible since there are only finitely many such hyperplanes, and put $\cH_{k+1} := \cH_k \cup \{v_{k+1}\}$, and then $\cH := \bigcup_{k} \cH_k$. Then any set of $r$ elements in $E$ are linearly independent over $\bR$; therefore they generate a finite index subgroup of $\bZ^r$. 
\end{proof}

\subsection{Random walks on $\Gamma$}
Throughout this section, we will make the standing assumption that $\Gamma < \GL_r (\bZ)$ is a strongly irreducible subgroup containing a proximal element. 

Let $p$ be a random walk on $\Gamma$. We say that $p$ is \emph{generating} if the support of $p$, denoted $\mbox{supp }p$, generates $\Gamma$. We then define the probability space $(\Omega, \bP)$ by taking $\Omega := \Gamma^{\bN}$, $\bP := p^{\otimes \bN }$. The expectation $\bE$ is taken with respect to $\bP$. We let $\gamma_n ((g_1, g_2, ...)) := g_1 \cdots g_n$ be the $n$:th step of the random walk, so that $\gamma_n$ has distribution $p^{*n}$. We will need to make use of the following results on random walks. We denote by $[ \cdot ]$ the projection map from $\bR^r$ to the projective plane $\bP (\bR^r)$.

\begin{lem}[\cite{BQ}, Lemma 4.6] \label{planezero}
    Let $p$ be a generating random walk on $\Gamma$. Then for any $p$-stationary measure $\nu$ on the projective plane $\bP(\bR^r)$, 
    \begin{equation*}
        \nu([L]) = 0
    \end{equation*}
    for every hyperplane $L \subseteq \bR^r$. 
\end{lem}

\begin{thm}[\cite{BLFM}, Corollary B.]\label{BLFM}
    Let $p$ be a generating random walk on $\Gamma$, and consider the natural action of $\Gamma$ on $\widehat{\Lambda}$. If $\tau$ is a probability measure on $\widehat{\Lambda}$ which satisfies $\tau(\Rat \mbox{ } \Lambda) = 0$, then 
    \begin{equation*}
        \frac{1}{N} \sum_{n = 1}^N p^{*n} * \tau \to m_{\widehat{\Lambda}}
    \end{equation*}
    weakly as $N \to \infty$, where $m_{\widehat{\Lambda}}$ is the Haar measure on $\widehat{\Lambda}$. 
\end{thm}

\subsection{Upper $(p, v)$-density}

\begin{defi}
    Let $p$ be a generating random walk on $\Gamma$. For a subset $E$ of $\Lambda$ we define the \emph{upper $(p, v)$-density} of $E$ as
    \begin{equation*}
        \overline{d}_{p, v} (E) := \limsup_N \frac{1}{N} \sum_{n=1}^N \bP(\gamma_n (v) \in E)
    \end{equation*}
\end{defi}
Note that $\overline{d}_{p, v} (E) = \overline{d}_{p, v} (E \cap \Gamma .v)$ for every $E$.

\begin{lem} \label{zerohyperplane}
    For every generating random walk $p$ on $\Gamma$ and every nonzero $v \in \Lambda$, 
    \begin{equation*}
        \overline{d}_{p, v} (L) = 0
    \end{equation*}
    for every hyperplane $L \subseteq \bR^r$.  
\end{lem}
\begin{proof}
    Any weak$^*$ - convergent subsequence of 
    \begin{equation*}
        \frac{1}{N} \sum_{n = 1}^N p^{*n} * \delta_{[v]} 
    \end{equation*}
    has a $p$-stationary measure $\nu$ as its limit. Since $[L]\subseteq \bP (\bR ^r )$ is a closed subset, for the limit along a convergent subsequence indexed by $k$ we get 
    \begin{equation*}
        \lim_{k \to \infty} \frac{1}{N_k} \sum_{n = 1}^{N_k} p^{*n} * \delta_{[v]} ([L]) \leq \nu ([L]) = 0.
    \end{equation*}
    By Lemma \ref{planezero} it follows that
    \begin{equation*}
        \begin{split}
            \overline{d}_{p, v} (L) = \limsup_{N \to \infty} \frac{1}{N} \sum_{n = 1}^N \bP (\gamma_n (v) \in L) \\ 
            = \limsup_{N \to \infty} \frac{1}{N} \sum_{n = 1}^N \bP (\gamma_n .[v] \in [L]) \\
            = \limsup_{N \to \infty} \frac{1}{N} \sum_{n = 1}^N p^{*n} * \delta_{[v]} ([L]) = 0. 
        \end{split}
    \end{equation*}
\end{proof}

Our goal now is to show that  
\begin{equation*}
    \overline{d}_{p, u} (\{v \in \Lambda: \mu(B \cap v.B) > c \} \cap \Gamma .v_0 ) > 0; 
\end{equation*}
in light of Lemma \ref{zerohyperplane} this will imply that the set in question is not contained in a finite union of hyperplanes. 

\begin{lem} \label{Ebig}
    For any generating random walk $p$ on $\Gamma$ and any  $v \in \Lambda \setminus \{ 0\}$, 
    \begin{equation*}
        \limsup_N |\frac{1}{N} \sum_{n=1}^N \bE[\mu(B \cap \gamma_n (v).B)] -\mu(B)^2 |\leq \sigma_B (\mbox{Rat }\Lambda \setminus \{0\} ). 
    \end{equation*}
\end{lem}
\begin{proof}
    By Bochner's theorem we can write
    \begin{equation*}
        \begin{split}
            \mu(B \cap \gamma_n(v) . B) = \int_{\widehat{\Lambda}} \langle \xi, \gamma_n (v)\rangle d \sigma_B (\xi) \\
            = \int_{\widehat{\Lambda}} \langle \gamma_n ^* (\xi), v\rangle d \sigma_B (\xi) = \int_{\widehat{\Lambda}} \langle \xi,v\rangle d (\gamma_{n*} ^* \sigma_B) (\xi).
        \end{split}
    \end{equation*}
    It follows that
    \begin{equation*}
        \bE[\mu(B \cap \gamma_n (v).B)] = \int_{\widehat{\Lambda}} \langle \xi,v\rangle d (p ^{*n} * \sigma_B) (\xi)
    \end{equation*}
    and
    \begin{equation*}
        \frac{1}{N} \sum_{n=1}^N \bE[\mu(B \cap \gamma_n (v).B)] = \int_{\widehat{\Lambda}} \langle \xi,v\rangle d ( \frac{1}{N} \sum_{n=1}^N p ^{*n} * \sigma_B) (\xi).
    \end{equation*}
    We now write
    \begin{equation*}
        \sigma_B = \mu(B) ^2 \delta_0 + \sigma_B ^{Rat} + \tau, 
    \end{equation*}
    where $\sigma_B ^{Rat}$ is supported on $\mbox{Rat }\Lambda \setminus \{0\}$ and $\tau(\mbox{Rat }\Lambda ) = 0$. 
    It follows from Theorem \ref{BLFM} that $\frac{1}{N}\sum_{n=1}^N p^n * \tau \to m_{\widehat{\Lambda}}$ in the weak$^*$ -topology. Also  $p *  \delta_0 = \delta_0$ and $\sigma_B ^{Rat} (\widehat{\Lambda}) = \sigma_B (\mbox{Rat} \mbox{ } \Lambda \setminus \{0\})$. Therefore 
    \begin{equation*}
        \begin{split}
            |\frac{1}{N} \sum_{n=1}^N \bE[\mu(B \cap \gamma_n (v).B)] -\mu(B)^2 | \leq |\int_{\widehat{\Lambda}}\langle \xi, v \rangle d ( \frac{1}{N} \sum_{n=1}^Np^{*n} * \sigma_B ^{Rat}) (\xi)| \\ + |\int_{\widehat{\Lambda}}\langle \xi, v \rangle d ( \frac{1}{N} \sum_{n=1}^N p^{*n} * \tau ) (\xi)| \leq \sigma_B (Rat \mbox{ }\Lambda \setminus \{0\} ) + |\int_{\widehat{\Lambda}}\langle \xi, v \rangle d ( \frac{1}{N} \sum_{n=1}^N p^{*n} * \tau ) (\xi)|.
        \end{split}
    \end{equation*}
    Since $\int_{\widehat{\Lambda}}\langle \xi, v \rangle d ( \frac{1}{N} \sum_{n=1}^N p^{*n} * \tau ) (\xi) \to \int_{\widehat{\Lambda}} \langle \xi, v \rangle dm_{\widehat{\Lambda}}(\xi) = 0$ for nontrivial $v$, the lemma follows.  
\end{proof}
Lemma \ref{Ebig} allows us to estimate $\bE [\mu (B \cap \gamma_n (v) .B)]$ from below. To obtain an estimate of $$\bP(\mu(B \cap \gamma_n (v).B) > c )$$ from this, we use the following Markov-type inequality. 
\begin{lem} \label{ineq}
    Let $f$ be a nonnegative bounded random variable. Then if $0 < c < \bE[f]$, we have that
    \begin{equation*}
        \bP( f > c) \geq \frac{\bE[f]-c}{\|f\|_{\infty} - c}. 
    \end{equation*}
\end{lem}
\begin{proof}
    We compute
    \begin{equation*}
        \begin{split}
            \bE[f] = \bE[f \chi_{f>c}] + \bE[f \chi_{f \leq c}] \leq \|f\|_\infty \bP(f > c) + c (1 - \bP(f> c)) \\
            = (\|f\|_\infty -c) \bP(f > c) + c
        \end{split}
    \end{equation*}
    which is equivalent to the claim.
\end{proof}
For the final arguments, we will make use of the \emph{asymptotic lower density} of a set $E \subseteq \bN$, defined as 
\begin{equation*}
    \underline{d} (E) := \liminf_{N}  \frac{|E \cap [0, N-1]|}{N}. 
\end{equation*}

\begin{lem} \label{Cdelta}
    Let $0 < c < \mu(B)^2 - \sigma_B (\Rat \mbox{ } \Lambda \setminus\{0\})$. Then there is $\delta > 0$ such that
    \begin{equation*}
        C := \{n \in \bN: \bP(\mu(B \cap \gamma_n (v).B) > c) > \delta \}
    \end{equation*}
    satisfies
    \begin{equation*}
        \underline{d} (C) > 0. 
    \end{equation*}
\end{lem}
\begin{proof}
    Let $0 < c < c' < \mu(B)^2 - \sigma_B (\Rat \mbox{ } \Lambda \setminus\{0\})$ and put
    \begin{equation*}
        A := \{n \in \bN: \bE[\mu (B \cap \gamma_n (v)B)] > c' \}. 
    \end{equation*}
    It follows from Lemma \ref{Ebig} that $\underline{d} (A) > 0$, and for any $n \in A$, using the inequality in Lemma \ref{ineq} applied to $f = \mu(B \cap \gamma_n (v) .B)$ (note here that $\|f\|_\infty \leq \mu(B)$), we obtain
    \begin{equation*}
        \begin{split}
            \bP(\mu(B \cap \gamma_n (v).B) > c)  \geq \frac{\bE[\mu(B \cap \gamma_n (v). B)] - c}{\mu(B) - c} \\
            \geq \frac{c' - c}{\mu(B) - c} > 0. 
        \end{split} 
    \end{equation*}
    Thus with $\delta := \frac{c' - c}{\mu(B) - c}$, $A \subseteq C$ and the claim is proved. 
\end{proof}

\begin{proof}[Proof of Lemma \ref{haystack}] Fix $v_0 \in \Lambda$.
    Put
    \begin{equation*}
        E := \{v \in \Lambda: \mu (B \cap v. B) > c \}
    \end{equation*}
    and fix a generating random walk $p$ on $\Gamma$. If $E$ were contained in a finite union $\bigcup_{k=1}^K L_k$ of hyperplanes $L_k \subseteq \bR ^r$, by Lemma \ref{haystackhyperplanes} we would get
    \begin{equation*}
        \overline{d}_{p, u} (E) \leq \sum_{k=1}^K \overline{d}_{p,u} (L_k) = 0.
    \end{equation*}
    Therefore, the Lemma follows if we can show that $\overline{d}_{p, v_0} (E) = \overline{d}_{p,v_0} (E \cap \Gamma.v_0) > 0$. We estimate
    \begin{equation*}
        \begin{split}
            \overline{d}_{p, v_0} (E) = \limsup_N \frac{1}{N} \sum_{n=1}^N \bP(\gamma_n (v_0) \in E) \\
        = \limsup_N \frac{1}{N} \sum_{n=1}^N \bP(\mu (B \cap \gamma_n (v_0) . B) > c ) \\
         \geq \limsup_N \frac{1}{N} \sum_{n=1}^N \delta \cdot \chi_C(n)
        \geq \delta \cdot\underline{d} (C) > 0, 
        \end{split}
    \end{equation*}
    where $C$ and $\delta$ are as in Lemma \ref{Cdelta}. 
\end{proof}

\bibliographystyle{amsplain}

\end{document}